\documentclass[11pt]{article}

\usepackage[margin=1.1in]{geometry}
\usepackage[utf8]{inputenc}
\usepackage[T1]{fontenc}
\usepackage{lmodern}
\usepackage{microtype}
\usepackage{amsmath,amssymb,amsthm,mathtools}
\usepackage{enumitem}
\usepackage[colorlinks=true,linkcolor=blue,citecolor=blue,urlcolor=blue]{hyperref}

\numberwithin{equation}{section}

\theoremstyle{plain}
\newtheorem{theorem}{Theorem}[section]
\newtheorem{corollary}[theorem]{Corollary}
\newtheorem{proposition}[theorem]{Proposition}
\newtheorem{lemma}[theorem]{Lemma}

\theoremstyle{definition}
\newtheorem{definition}[theorem]{Definition}

\theoremstyle{remark}

\newcommand{\R}{\mathbb{R}}

\newcommand{\diam}{\operatorname{diam}}
\newcommand{\osc}{\operatorname{osc}}
\newcommand{\sgn}{\operatorname{sign}}

\title{Hereditary Diameter Rigidity in Real $L_1$-Spaces}
\author{Rafael Chiclana%
	\thanks{Department of Mathematics, Michigan State University, United States.
		E-mail: \texttt{chiclan1@msu.edu}}}
\date{}

\begin{document}
\maketitle

\begin{abstract}
	We prove that hereditary diameter lower bounds are stable under convex
	combinations in real \(L_1\)-spaces over arbitrary measure spaces. More
	precisely, if every weakly open piece of each set has diameter at least
	\(\delta>0\), then every finite or countable convex combination has the
	corresponding hereditary lower bound \(\delta/4\). In the diameter-two case,
	the bound is preserved exactly. Consequently, for every topology containing
	the relative weak topology, the diameter two and strong diameter two
	properties are equivalent on bounded convex subsets of the unit ball, as
	are the convex point of continuity property and strong regularity.
\end{abstract}
\section{Introduction}

The diameter two properties of Banach spaces measure how large the natural
weak pieces of the unit ball are; see, for instance,
\cite{AbrahamsenLimaNygaard2013,BecerraLopezRueda2014}. More generally, let
\(D\) be a non-empty bounded convex subset of a Banach space and let
\(\tau\) be a topology on \(D\). We say that \(D\) has the
\emph{\(\tau\)-diameter two property} if every non-empty \(\tau\)-open
subset of \(D\) has diameter two, and that \(D\) has the
\emph{\(\tau\)-strong diameter two property} if every finite convex
combination of non-empty \(\tau\)-open subsets of \(D\) has diameter two.
When \(D=B_X\) and \(\tau\) is the relative weak topology, these are the
usual diameter two and strong diameter two properties of \(X\), and we omit
\(\tau\) from the terminology.

In general, the diameter two property does not imply the strong diameter two
property. For \(L_1(\mu)\), however, the diameter two and strong diameter two
properties are equivalent for the usual weak topology. Indeed, the strong
diameter two property holds precisely when \(\mu\) has no atom of finite
measure \cite[Proposition~4.12]{MartinPerreauRuedaZoca2024}, while a
finite-measure atom yields slices of arbitrarily small diameter.

Let \(\tau\) be a topology on \(D\) containing its relative weak topology. For
\(C\subset D\), let \(\tau_C\) denote the topology induced by \(\tau\) on
\(C\). Following L\'opez-P\'erez and Medina
\cite{LopezPerezMedina2024}, we say that \(D\) has the
\emph{\(\tau\)-convex point of continuity property}, or
\(\tau\)-CPCP, if every non-empty bounded convex subset \(C\subset D\)
contains non-empty \(\tau_C\)-open subsets of arbitrarily small diameter.
We say that \(D\) is \emph{\(\tau\)-strongly regular}, or \(\tau\)-SR, if
every such \(C\) contains finite convex combinations of non-empty
\(\tau_C\)-open subsets having arbitrarily small diameter. Clearly,
\(\tau\)-CPCP implies \(\tau\)-SR, since a single open set may be regarded
as a convex combination with one term. We say that \(D\) is \emph{\(\tau\)-strongly regular for open subsets} if
every non-empty convex \(\tau\)-open subset \(O\subset D\) contains finite
convex combinations of non-empty \(\tau_O\)-open subsets having
arbitrarily small diameter. This property is formally weaker than
\(\tau\)-strong regularity.

L\'opez-P\'erez and Medina \cite{LopezPerezMedina2024} study these notions
for locally convex topologies containing the relative weak topology. In the
separable setting, and more generally under a countable-tightness
hypothesis, they show that the Krein--Milman property forces
\(\tau\)-CPCP and \(\tau\)-strong regularity to be equivalent; see
\cite[Theorem~2.6 and Corollary~2.9]{LopezPerezMedina2024}. This suggests
a possible approach to the problem of whether the Krein--Milman property
implies the Radon--Nikod\'ym property: construct a topology for which
strong regularity holds while CPCP fails.

As a first step, L\'opez-P\'erez and Medina construct a locally convex
topology \(\tau\) on \(B_{c_0}\), containing the relative weak topology,
such that \(B_{c_0}\) has the \(\tau\)-diameter two property and is
\(\tau\)-strongly regular for open subsets. Thus every non-empty
\(\tau\)-open subset has diameter two, while every non-empty convex \(\tau\)-open subset \(O\) contains finite convex combinations of non-empty \(\tau_O\)-open subsets of arbitrarily small diameter. In particular, \(B_{c_0}\) fails both \(\tau\)-CPCP and the
\(\tau\)-strong diameter two property. It is therefore natural to ask
whether this simultaneous behavior can occur on \(B_{L_1(\mu)}\).

In this note, we establish a hereditary diameter rigidity phenomenon for real
\(L_1\)-spaces over arbitrary measure spaces. As a consequence, the
topological behavior described above cannot occur on \(B_{L_1(\mu)}\).
More generally, for every topology \(\tau\) containing the relative weak
topology, the \(\tau\)-diameter two and \(\tau\)-strong diameter two
properties are equivalent, and the \(\tau\)-convex point of continuity
property is equivalent to \(\tau\)-strong regularity on bounded convex
subsets of \(B_{L_1(\mu)}\).

We introduce the following terminology. Let \(0<\delta\leq2\). A non-empty
set \(C\subset B_X\) is said to be \emph{weakly hereditarily diameter
	\(\delta\)}, or wHD-\(\delta\), if every non-empty weakly open piece of
\(C\) has diameter at least \(\delta\); see
Definition~\ref{def:wHD-delta}. In the endpoint case \(\delta=2\), we write
wHD2. For a finite or countable family of subsets \(C_i\subset B_X\) and
coefficients \(\lambda_i\geq0\) with \(\sum_i\lambda_i=1\), we write
\[
\sum_i\lambda_i C_i
=
\left\{
\sum_i\lambda_i c_i:c_i\in C_i
\right\}.
\]
In the countable case, the series is interpreted in the norm topology. Since
\(C_i\subset B_X\), every such series converges absolutely in norm.

Our first main result gives a quantitative hereditary lower bound for
countable convex combinations.

\begin{theorem}\label{thm:main-delta}
Let $(\Omega,\Sigma,\mu)$ be an arbitrary measure space and set
$X=L_1(\mu;\R)$. Let $0<\delta\leq 2$, let $(C_i)_{i=1}^\infty$ be non-empty wHD-$\delta$ subsets of $B_X$,
and let $\lambda_i\geq 0$ with
$\sum_{i=1}^\infty\lambda_i=1$. Then
\[\sum_{i=1}^\infty \lambda_i C_i\] 
is wHD-$(\delta/4)$.
\end{theorem}

The preceding theorem provides the universal constant \(1/4\), which is not
asserted to be optimal. At the endpoint \(\delta=2\), however, a sharper
argument gives the optimal result: diameter two is preserved exactly.

\begin{theorem}\label{thm:main-2}
	Let $(\Omega,\Sigma,\mu)$ be an arbitrary measure space and set
	$X=L_1(\mu;\R)$. Let $(C_i)_{i=1}^\infty$ be non-empty wHD2 subsets of $B_X$,
	and let $\lambda_i\geq 0$ with
	$\sum_{i=1}^\infty\lambda_i=1$. Then
	\[\sum_{i=1}^\infty \lambda_i C_i\] 
	is wHD2. In particular, it has diameter $2$.
\end{theorem}

The two main theorems have the following topological interpretation. They
show that, on the unit ball of a real \(L_1\)-space, lower diameter bounds
for open sets cannot disappear under convex combinations. Quantitatively,
a uniform lower bound \(\delta\) is preserved up to the universal factor
\(1/4\), while at the endpoint \(\delta=2\) no loss occurs.

\begin{corollary}\label{cor:topological-main}
	Let \((\Omega,\Sigma,\mu)\) be an arbitrary measure space and set
	\(X=L_1(\mu;\mathbb R)\). Let \(D\subset B_X\) be non-empty, and let
	\(\tau\) be a topology on \(D\) containing the relative weak topology.
	
	\begin{enumerate}[label=\textup{(\roman*)}]
		\item Let \(0<\delta\leq2\). If every non-empty \(\tau\)-open subset of
		\(D\) has diameter at least \(\delta\), then every finite or countable
		convex combination of non-empty \(\tau\)-open subsets of \(D\) is
		wHD-\((\delta/4)\).
		
		\item If every non-empty \(\tau\)-open subset of \(D\) has diameter two,
		then every finite or countable convex combination of non-empty
		\(\tau\)-open subsets of \(D\) is wHD2.
	\end{enumerate}
\end{corollary}

The preceding geometric statement yields the three topological consequences
that motivate this work.

\begin{corollary}\label{cor:topological-consequences}
	Let \((\Omega,\Sigma,\mu)\) be an arbitrary measure space, set
	\(X=L_1(\mu;\mathbb R)\), and let \(D\subset B_X\) be a non-empty
	convex set. Let \(\tau\) be a topology on \(D\) containing the relative weak
	topology.
	
	\begin{enumerate}[label=\textup{(\roman*)}]
		\item \(D\) has the \(\tau\)-diameter two property if and only if it has the
		\(\tau\)-strong diameter two property.
		
		\item \(D\) has the \(\tau\)-CPCP if and only if it is
		\(\tau\)-strongly regular.
		\item If there exists \(0<\delta\leq 2\) such that every non-empty
		\(\tau\)-open subset of \(D\) has diameter at least \(\delta\), then
		\(D\) is not \(\tau\)-strongly regular for open subsets.
	\end{enumerate}
\end{corollary}

Thus, on bounded convex subsets of a real \(L_1\)-space, no topology
containing the relative weak topology can separate the diameter two property
from the strong diameter two property, or the convex point of continuity
property from strong regularity. Moreover, a uniform positive lower bound
on the diameters of open sets is incompatible with strong regularity for
open subsets. No local convexity assumption on
\(\tau\) is needed.

The rest of the paper is organized as follows.
Section~\ref{sec:2} collects the notation and preliminary facts.
Section~\ref{sec:roadmap} presents a common proof scheme and proves the main
results, assuming two auxiliary weak-star density propositions.
Section~\ref{sec:auxiliary} establishes these propositions through two
\(L_1\)-localization arguments.

\section{Preliminaries}\label{sec:2}

Throughout the paper, all Banach spaces are over the real scalar field. Let $(\Omega,\Sigma,\mu)$ be an arbitrary measure space and set
$X=L_1(\mu;\mathbb{R})$. We write $B_X$ for the closed unit ball of $X$ and
$B_{L_\infty(\mu)}$ for the closed unit ball of $L_\infty(\mu;\R)$. For convenience, we write simply $L_1(\mu)$ and $L_\infty(\mu)$. For $\varphi\in L_\infty(\mu)$ and $f\in L_1(\mu)$ we write
\[
        \varphi(f)=\int_\Omega \varphi f\,d\mu.
\]
This gives a bounded linear functional on $X$ of norm at most
$\|\varphi\|_\infty$, yielding a canonical contractive linear operator
\[
        J:L_\infty(\mu)\longrightarrow X^*,\qquad J\varphi(f)=\int_\Omega
        \varphi f\,d\mu.
\]

\begin{lemma}\label{lem:Linfty-wstar-dense}
The set \(J(B_{L_\infty(\mu)})\) is weak-star dense in \(B_{X^*}\).
\end{lemma}

\begin{proof}
For every $f\in X$,
\[
        \|f\|_1=\sup_{\varphi\in B_{L_\infty(\mu)}} \left|\int_\Omega
        \varphi f\,d\mu\right|,
\]
by taking $\varphi=\sgn(f)$, with $\sgn(0)=0$. Hence $J(B_{L_\infty(\mu)})$ is
norming for $X$. It is convex and balanced. Its polar in $X$ is therefore
$B_X$, and the bipolar theorem gives
\[
        \overline{J(B_{L_\infty(\mu)})}^{\omega^*}=B_{X^*}.\qedhere
\]
\end{proof}
For a non-empty bounded subset \(C\) of a Banach space \(Y\), we write
\[
\diam(C):=\sup\{\|x-y\|:x,y\in C\}.
\]

\begin{definition}\label{def:wHD-delta}
Let $Y$ be a Banach space, let $C\subset B_Y$ be non-empty, and let $0<\delta\leq 2$. We say
that $C$ is \emph{weakly hereditarily diameter $\delta$}, or wHD-$\delta$, if for every weakly open set $W\subset Y$ such that $C\cap W\ne\emptyset$, one
has
\[
        \diam(C\cap W)\geq \delta.
\]
In the case $\delta=2$, we say $C$ is wHD2.
\end{definition}

The following simple observation shows that the class of weakly hereditarily
diameter $\delta$ sets is stable under intersections with weakly open sets.

\begin{lemma}\label{lem:weak-piece-wHD-delta}
	Let \(Y\) be a Banach space and $0<\delta\leq 2$. If \(C\subset B_Y\) is wHD-$\delta$ and
	\(V\subset Y\) is weakly open with \(C\cap V\ne\emptyset\), then
	\(C\cap V\) is wHD-$\delta$.
\end{lemma}

\begin{proof}
	Let \(W\subset Y\) be weakly open and suppose
	\((C\cap V)\cap W\ne\emptyset\). Then \(V\cap W\) is weakly open in \(Y\)
	and meets \(C\). Since \(C\) is wHD-$\delta$,
	\[
	\diam((C\cap V)\cap W)=\diam\big(C\cap(V\cap W)\big)\geq \delta.\qedhere
	\]
\end{proof}

Let $C\subset B_X$. We define the \emph{oscillation} of a functional $\varphi\in B_{X^*}$ on
$C$ by
\[
        \osc_C(\varphi)
        =\sup\{\varphi(f-g):f,g\in C\}.
\]

\section{Roadmap and proof of the main results}\label{sec:roadmap}

The proofs of Theorems~\ref{thm:main-delta} and~\ref{thm:main-2} follow the same three-step approach. First, a diameter formula
reduces the diameter of a countable convex combination to a weighted sum of
oscillations. Second, we prove two weak-star density results for functionals with large
oscillation: a quantitative result for wHD-\(\delta\) sets and a sharper
endpoint result for wHD2 sets. Finally, Baire's theorem gives one functional with
simultaneously large oscillation on countably many sets.

\begin{lemma}\label{lem:diameter-formula}
Let $Y$ be a Banach space. Let $(C_i)_{i=1}^\infty$ be non-empty subsets
of $B_Y$, and let $\lambda_i\geq 0$ with
$\sum_{i=1}^\infty\lambda_i=1$. Then
\[
        \diam \left ( \sum_{i=1}^\infty \lambda_i C_i \right )
        =\sup_{\psi\in B_{Y^*}}
          \sum_{i=1}^\infty \lambda_i\osc_{C_i}(\psi).
\]
\end{lemma}

\begin{proof}
	For every choice \(u_i,v_i\in C_i\), the series $\sum_{i=1}^\infty \lambda_i(u_i-v_i)$ converges in norm. Hence
	\[
	\begin{aligned}
		\diam\left (\sum_{i=1}^\infty \lambda_i C_i\right )
		&=
		\sup_{u_i,v_i\in C_i}
		\left\|
		\sum_{i=1}^\infty \lambda_i(u_i-v_i)
		\right\|  \\
		&=
		\sup_{u_i,v_i\in C_i}
		\sup_{\psi\in B_{Y^*}}
		\psi\left(\sum_{i=1}^\infty \lambda_i(u_i-v_i)\right) \\
		&=
		\sup_{\psi\in B_{Y^*}}
		\sup_{u_i,v_i\in C_i}
		\sum_{i=1}^\infty \lambda_i\psi(u_i-v_i) \\
		&=
		\sup_{\psi\in B_{Y^*}}
		\sum_{i=1}^\infty
		\lambda_i
		\sup_{u_i,v_i\in C_i}\psi(u_i-v_i) \\
		&=
		\sup_{\psi\in B_{Y^*}}
		\sum_{i=1}^\infty \lambda_i\osc_{C_i}(\psi).
	\end{aligned}
	\]
	To justify the fourth equality, fix \(\psi\in B_{Y^*}\) and write $a_i=\osc_{C_i}(\psi)$. For every choice \(u_i,v_i\in C_i\),
	\[
	\sum_{i=1}^{\infty}\lambda_i\psi(u_i-v_i)
	\leq
	\sum_{i=1}^{\infty}\lambda_i a_i.
	\]
	Conversely, fix \(n\in\mathbb N\) and \(\varepsilon>0\). For
	\(1\leq i\leq n\), choose \(u_i,v_i\in C_i\) so that
	\[
	\psi(u_i-v_i)>a_i-\varepsilon.
	\]
	For \(i>n\), choose \(u_i=v_i\in C_i\). Then
	\[
	\sup_{u_i,v_i\in C_i}
	\sum_{i=1}^{\infty}\lambda_i\psi(u_i-v_i)
	\geq
	\sum_{i=1}^{n}\lambda_i a_i-\varepsilon.
	\]
	Letting \(\varepsilon\downarrow0\) and then \(n\to\infty\) gives the reverse
	inequality, since \(0\leq a_i\leq2\) and \(\sum_i\lambda_i=1\).
\end{proof}

In view of Lemma~\ref{lem:diameter-formula}, to lower bound the diameter of convex
combinations, it is enough to find a single functional which has large
oscillation on each of the sets involved. For \(C\subset B_X\) and \(0<\alpha<2\),
write
\[
G(C,\alpha)=\{x^*\in B_{X^*}:\osc_C(x^*)>\alpha\}.
\]
The following two density results are the key inputs. The first gives the
positive-diameter estimate needed for the wHD-\(\delta\) theorem, while the
second gives the sharp endpoint estimate in the case \(\delta=2\). Both are
proved in Section~\ref{sec:auxiliary}.

\begin{proposition}\label{prop:G-delta-dense}
	Let \(0<\delta\leq2\), let \(C\subset B_X\) be wHD-\(\delta\), and let $0<\alpha<\delta/4$. Then \(G(C,\alpha)\) is weak-star open and weak-star dense in \(B_{X^*}\).
	\end{proposition}

\begin{proposition}\label{prop:G-two-dense}
	Let \(C\subset B_X\) be wHD2 and let $0<\alpha<2$. Then \(G(C,\alpha)\) is
	weak-star open and weak-star dense in \(B_{X^*}\).
\end{proposition}

The Baire category theorem then allows us to choose one functional with
large oscillation on countably many wHD-$\delta$ sets simultaneously.

\begin{lemma}\label{lem:common-functional}
	Let \((C_i)_{i=1}^\infty\) be non-empty subsets of \(B_X\), and let
	\(0<\alpha<2\). Assume that \(G(C_i,\alpha)\) is weak-star open and
	weak-star dense in \(B_{X^*}\) for every \(i\geq1\). Then there exists
	\(x^*\in B_{X^*}\) such that
	\[
	\osc_{C_i}(x^*)>\alpha
	\quad \forall\, i\geq1.
	\]
\end{lemma}

\begin{proof}
	Since \(B_{X^*}\) is weak-star compact, it is a Baire space. Therefore
	\[
	\bigcap_{i=1}^\infty G(C_i,\alpha)
	\]
	is dense in \(B_{X^*}\), and in particular non-empty. Any element of this
	intersection has the desired property.
\end{proof}

We are now ready to prove the main results.

\begin{proof}[Proof of Theorem~\ref{thm:main-delta}]
	For every choice $c_i\in C_i$, the series $\sum_i\lambda_ic_i$ converges in
	norm because
	\[
	\sum_{i=1}^\infty\lambda_i\|c_i\|
	\leq \sum_{i=1}^\infty\lambda_i=1.
	\]
	Thus
	\[
	S=\sum_{i=1}^\infty\lambda_iC_i
	\]
	is well-defined and contained in $B_X$. We first prove that $\diam(S)\geq \delta/4$. Let
	$0<\alpha<\delta/4$. By Proposition \ref{prop:G-delta-dense} and Lemma \ref{lem:common-functional}, there is $x^*\in B_{X^*}$ such that
	\[
	\osc_{C_i}(x^*)>\alpha
	\quad \forall \, i\geq 1.
	\]
	By Lemma~\ref{lem:diameter-formula},
	\[
	\diam(S)
	\geq \sum_{i=1}^\infty\lambda_i\osc_{C_i}(x^*)
	> \sum_{i=1}^\infty\lambda_i\alpha=\alpha.
	\]
	Since $0<\alpha<\delta/4$ was arbitrary, this gives $\diam(S)\geq\delta/4$.
	
	It remains to prove the hereditary statement. Let $W\subset X$ be weakly open with $S\cap W\ne\emptyset$. Choose
	\[
	s_0=\sum_{i=1}^\infty\lambda_i c_i\in S\cap W,
	\qquad c_i\in C_i.
	\]
	There are $x_1^*,\ldots,x_m^*\in X^*$ and $\eta>0$ such that the basic weak neighborhood
	\[
	V=\left\{x\in X:
	|x_\ell^*(x-s_0)|<\eta,
	\ 1\leq \ell\leq m\right\}
	\]
	satisfies $s_0\in V\subset W$. For each $i$, define
	\[
	V_i=\left\{x\in X:
	|x_\ell^*(x-c_i)|<\eta,
	\ 1\leq \ell\leq m\right\},
	\qquad D_i=C_i\cap V_i.
	\]
	Then $D_i$ is non-empty and wHD-$\delta$ by Lemma~\ref{lem:weak-piece-wHD-delta}. If
	$d_i\in D_i$ for all $i$, then for every $\ell$,
	\begin{equation}\label{eq:D_i}
		\left|x_\ell^*\left(\sum_{i=1}^\infty\lambda_i d_i-s_0\right)\right|
		\leq \sum_{i=1}^\infty\lambda_i |x_\ell^*(d_i-c_i)| <\eta.
	\end{equation}
	Hence
	\[
	\sum_{i=1}^\infty\lambda_iD_i\subset S\cap V\subset S\cap W.
	\]
	By the diameter assertion already proved, applied to the sequence
	$(D_i)_{i\geq1}$,
	\[
	\diam\left(\sum_{i=1}^\infty\lambda_iD_i\right)\geq \delta/4.
	\]
	Consequently $\diam(S\cap W)\geq \delta/4$, and thus $S$ is wHD-$(\delta/4)$.
\end{proof}

\begin{proof}[Proof of Theorem~\ref{thm:main-2}]
	We follow the proof of Theorem~\ref{thm:main-delta}, replacing
	Proposition~\ref{prop:G-delta-dense} by
	Proposition~\ref{prop:G-two-dense}. Let
	\[
	S=\sum_{i=1}^\infty \lambda_i C_i .
	\]
	As before, \(S\) is well-defined and \(S\subset B_X\). We first prove that
	\(\diam(S)=2\). Let \(0<\alpha<2\). By
	Proposition~\ref{prop:G-two-dense} and
	Lemma~\ref{lem:common-functional}, there exists \(x^*\in B_{X^*}\) such that
	\[
	\osc_{C_i}(x^*)>\alpha
	\quad \forall\, i\geq1.
	\]
	Hence, by Lemma~\ref{lem:diameter-formula},
	\[
	\diam(S)
	\geq
	\sum_{i=1}^\infty \lambda_i\osc_{C_i}(x^*)
	>
	\alpha.
	\]
	Since \(0<\alpha<2\) was arbitrary, \(\diam(S)\geq2\). Since \(S\subset B_X\),
	we conclude that \(\diam(S)=2\).
	
	It remains to prove the hereditary statement. Let \(W\subset X\) be weakly
	open with \(S\cap W\ne\emptyset\). Choose \(s_0=\sum_i\lambda_i c_i\in S\cap W\)
	and define the weak neighborhoods \(V_i\) and the sets $D_i=C_i\cap V_i$ as in the proof of Theorem~\ref{thm:main-delta}. Then each \(D_i\) is wHD2 by
	Lemma~\ref{lem:weak-piece-wHD-delta}, and (\ref{eq:D_i}) gives
	\[
	\sum_{i=1}^\infty\lambda_iD_i\subset S\cap W.
	\]
	Applying the diameter assertion just proved to the sequence \((D_i)_{i\geq1}\),
	we get
	\[
	\diam\left(\sum_{i=1}^\infty\lambda_iD_i\right)=2.
	\]
	Therefore \(\diam(S\cap W)=2\). Since \(W\) was arbitrary, \(S\) is wHD2.
\end{proof}

\begin{proof}[Proof of Corollary~\ref{cor:topological-main}]
	Let \(U\subset D\) be a non-empty \(\tau\)-open set. Suppose first that every non-empty \(\tau\)-open subset of \(D\) has
	diameter at least \(\delta\). Let \(W\subset X\) be weakly open and assume
	that \(U\cap W\ne\emptyset\). Since \(\tau\) contains the relative weak
	topology, \(D\cap W\) is \(\tau\)-open in \(D\). Therefore
	\[
	U\cap W=U\cap(D\cap W)
	\]
	is a non-empty \(\tau\)-open subset of \(D\), and hence
	\[
	\diam(U\cap W)\geq\delta.
	\]
	Thus every non-empty \(\tau\)-open subset of \(D\) is wHD-\(\delta\).
	Part~\textup{(i)} now follows from
	Theorem~\ref{thm:main-delta}.
	
	If every non-empty \(\tau\)-open subset of \(D\) has diameter two, the same
	argument shows that every such subset is wHD2. Part~\textup{(ii)} therefore
	follows from Theorem~\ref{thm:main-2}.
\end{proof}

\begin{proof}[Proof of Corollary~\ref{cor:topological-consequences}]
	The implication from the \(\tau\)-strong diameter two property to the
	\(\tau\)-diameter two property is immediate, since a single non-empty
	\(\tau\)-open set is a convex combination with one term. Conversely, assume
	that \(D\) has the \(\tau\)-diameter two property. By
	Corollary~\ref{cor:topological-main}\textup{(ii)}, every finite convex
	combination of non-empty \(\tau\)-open subsets of \(D\) is wHD2 and, in
	particular, has diameter two. Hence \(D\) has the
	\(\tau\)-strong diameter two property. This proves \textup{(i)}.
	
	The implication from \(\tau\)-CPCP to \(\tau\)-strong regularity is also
	immediate. We prove the converse by contraposition. Suppose that \(D\) fails
	the \(\tau\)-CPCP. Then there exist a non-empty bounded convex subset
	\(C\subset D\) and a number \(\delta>0\) such that every non-empty
	 \(\tau_C\)-open subset has diameter at least \(\delta\), where \(\tau_C\) denotes the topology induced by \(\tau\) on \(C\). Since
	\(\tau\) contains the relative weak topology on \(D\), the topology
	\(\tau_C\) contains the relative weak topology on \(C\). Applying
	Corollary~\ref{cor:topological-main}\textup{(i)} to \(C\), we conclude that
	every finite convex combination of non-empty \(\tau_C\)-open
	subsets has diameter at least \(\delta/4\). Consequently, \(C\)
	does not contain such convex combinations of arbitrarily small diameter.
	Thus \(D\) is not \(\tau\)-strongly regular. This proves \textup{(ii)}.
	
	Finally, suppose that there exists \(\delta>0\) such that every non-empty
	\(\tau\)-open subset of \(D\) has diameter at least \(\delta\). By
	Corollary~\ref{cor:topological-main}\textup{(i)}, every finite convex
	combination of non-empty \(\tau\)-open subsets of \(D\) has diameter at
	least \(\delta/4\).
	
	If \(D\) were \(\tau\)-strongly regular for open subsets, then applying this
	property to the non-empty convex \(\tau\)-open set \(D\) itself would yield
	finite convex combinations of non-empty \(\tau\)-open subsets of \(D\) of
	arbitrarily small diameter. This contradicts the preceding lower bound.
	Thus \(D\) is not \(\tau\)-strongly regular for open subsets. This proves
	\textup{(iii)}.
\end{proof}

\section{Auxiliary results}\label{sec:auxiliary}

This section proves Propositions~\ref{prop:G-delta-dense} and
\ref{prop:G-two-dense}. The key observation is that if differences of points of \(C\) can be localized on sets which are small for arbitrary \(L_1\)-control functions, then
functionals with large oscillation on \(C\) are weak-star dense.

\begin{lemma}\label{lem:localization-to-density}
	Let \(C\subset B_X\), and let \(0<\alpha<2\). Assume that for every
	\(w\in L_1(\mu)\) and every \(\eta>0\), there exist
	\(y,z\in C\) and a measurable set \(E\subset\Omega\) such that
	\[
	\int_E |w|\,d\mu<\eta
	\qquad\text{and}\qquad
	\int_E |y-z|\,d\mu>\alpha.
	\]
	Then \(G(C,\alpha)\) is weak-star open and weak-star dense in \(B_{X^*}\).
\end{lemma}

\begin{proof}
	First we prove openness. Let $x_0^*\in G(C,\alpha)$. Choose
	$y,z\in C$ such that
	\[
	x^*_0(y-z)>\alpha.
	\]
	The map $x^*\mapsto x^*(y-z)$ is continuous for the weak-star topology
	on $B_{X^*}$. Hence
	\[
	\{x^*\in B_{X^*}: x^*(y-z)>\alpha\}
	\]
	is a weak-star neighborhood of $x_0^*$ contained in
	$G(C,\alpha)$.
	
	We now prove density. Let \(U\subset B_{X^*}\) be a non-empty weak-star open
	set. By Lemma~\ref{lem:Linfty-wstar-dense}, there is
	\(\varphi_0\in B_{L_\infty(\mu)}\) such that \(J\varphi_0\in U\). Then there exists a basic neighborhood $V \subset U$ of the form
	\[ V=
	\left\{
	x^*\in B_{X^*}:
	|(x^*-J\varphi_0)(g_i)|<\eta,\ 1\leq i\leq n
	\right\},\]
	where \(g_1,\dots,g_n\in L_1(\mu)\) and \(\eta>0\). Define $w:=\sum_{i=1}^n |g_i| \in L_1(\mu)$ and let $\nu(A):=\int_A w\,d\mu$. By hypothesis, there are $y$, $z \in C$ and $E\subset \Omega$ such that
	\[
	\nu(E)<\frac{\eta}{2} \qquad \text{and} \qquad \int_{E} |y-z|\, d\mu > \alpha.
	\]
	Define
	\[
	\varphi_\pm
	:=
	\varphi_0\mathbf 1_{E^c}
	\pm \sgn(y-z)\mathbf 1_E.
	\]
	Then $\varphi_+$, $\varphi_- \in B_{L_\infty(\mu)}$. Moreover, for every $i=1,\ldots,n$ we
	have
	\[
	\left|\int(\varphi_{+}-\varphi_0)g_i \,d\mu\right|
	\leq 2\int_{E} |g_i|\,d\mu
	\leq 2\nu(E)<\eta.
	\]
	Thus $J\varphi_+ \in V \subset U$. Similarly, $J \varphi_- \in V \subset U$. Finally,
	\begin{align*}
		J\varphi_+(y-z) - J\varphi_-(y-z) &= 2\int_E |y-z|\,d\mu>2\alpha.
	\end{align*}
	Hence either $\osc_C(J\varphi_+)>\alpha$ or $\osc_C(J\varphi_-)>\alpha$, and thus
	$U\cap G(C,\alpha)\ne\emptyset$. Since $U$ was an arbitrary weak-star open set, we deduce that
	$G(C,\alpha)$ is weak-star dense.
\end{proof}

\subsection{\texorpdfstring{\(L_1\)-localization for diameter two}
	{L1-localization for diameter two}}

This subsection is devoted to proving Proposition \ref{prop:G-two-dense}. It will follow from Lemma \ref{lem:localization-to-density} once we establish the appropriate $L_1$-localization principle. To start, we show that one can construct many points at almost maximal distance inside wHD2 sets.

\begin{lemma}\label{lem:diametral}
Let $Y$ be a Banach space and let $C\subset B_Y$ be wHD2. For every
integer $N\geq2$ and every $\beta>0$, there exist $f_1,\ldots,f_N\in C$ such
that
\[
        \|f_j-f_k\|>2-\beta \quad \forall\, j\ne k.
\]
\end{lemma}

\begin{proof}
Fix $\beta>0$. We prove this by induction on $N$. First, assume that $N=2$.
Since $\diam(C)=2$ there are $f_1$, $f_2 \in C$ with
$\|f_1-f_2\|>2-\beta$. Next, assume there exist $f_1,\ldots,f_N\in C$ such
that
\[
\|f_j-f_k\|>2-\beta \quad \forall \, j\ne k.
\]
For $j=1,\ldots,N-1$ define
$W_j=\{g \in Y : \|g-f_j\|>2-\beta\}$, which are weakly open since
$\|\cdot\|$ is weakly lower-semicontinuous. Thus,
\[ W:=\bigcap_{j=1}^{N-1} W_j\]
is weakly open. Moreover, $f_N \in W$ and thus $W\cap C\neq \emptyset$.
Since $C$ is wHD2, we have that $\diam(C\cap W)=2$. Consequently, there are
$g_1$, $g_2 \in C\cap W$ with $\|g_1-g_2\|>2-\beta$. The set
$\{f_1,\ldots,f_{N-1},g_1,g_2\}$ satisfies the desired condition.
\end{proof}

The idea behind the desired $L_1$-localization principle is that points in \(B_{L_1(\mu)}\) which are almost at distance two cannot have much same-sign overlap. Thus, in a large pairwise
almost diametral family, the regions where the functions are large must be
spread out. One may
think of these large-value regions as spikes which are almost disjoint on
average. We provide quantitative estimates for this phenomenon, relative to an arbitrary finite control
measure of the form
\[
\nu(A)=\int_A |w|\,d\mu,\qquad w\in L_1(\mu).
\]

The next lemma records the elementary overlap estimate behind the argument.
If two functions in the unit ball of \(L_1\) are almost at distance two, then
their positive parts have very small overlap, and the same is true of their
negative parts.

\begin{lemma}\label{lem:loc_1}
Let $f,g\in B_{L_1(\mu)}$ with $\|f-g\|_1>2-\beta$. Then
\[
        \int \min(f^+,g^+)\,d\mu+
        \int \min(f^-,g^-)\,d\mu < \frac{\beta}{2},
\]
where $f^+=\max\{f,0\}$ and $f^-=\max\{-f,0\}$.
\end{lemma}

\begin{proof}
For real numbers $a,b$, it is straightforward to see that
	\[
	|a-b|=|a|+|b|-2\min(a^+,b^+)-2\min(a^-,b^-),
	\]
	where $a^+=\max\{a,0\}$ and $a^-=\max\{-a,0\}$. Integrating gives
\[
\|f-g\|_1
=\|f\|_1+\|g\|_1
-2\int\min(f^+,g^+)\,d\mu
-2\int\min(f^-,g^-)\,d\mu.
\]
Since $\|f\|_1,\|g\|_1\leq1$, the assumption
$\|f-g\|_1>2-\beta$ gives the estimate.
\end{proof}

Next we reduce the control measure to a finite-measure region on which its
density is bounded.

\begin{lemma}\label{lem:loc_2}
Let $\nu(A)=\int_A |w| \,d\mu$, where $w\in L_1(\mu)$. For every $\eta>0$
there are a measurable set $H\subset\Omega$ with $\mu(H)<\infty$ and a
constant $M\geq1$ such that
\[
        \nu(H^c)<\eta
        \quad \text{and} \quad
        \nu(A)\leq M\mu(A)\quad \forall\, A\subset H\text{ measurable}.
\]
\end{lemma}

\begin{proof} For $n \in \mathbb{N}$, let $H_n=\{ n^{-1} < |w| < n\}$ and
define $w_n=|w| \mathbf{1}_{H_n}$. Then, $w_n \uparrow |w|$ and thus by the
monotone convergence theorem, there is $m \in \mathbb{N}$ such that
	\[ \int |w| \, d\mu - \int w_m \, d\mu = \int_{H_m^c} |w| \, d \mu < \eta.\]
Set $H=H_m$ and $M=m$. Then $\nu(H^c)<\eta$ and
\[
        \mu(H)= \int_H 1 \, d\mu \leq M \int_H |w|\, d\mu \leq M\|w\|_1 <\infty.
\]
Finally, for any measurable $A \subset H$ we have
\[
        \nu(A)=\int_A |w| \,d\mu\leq M\mu(A).\qedhere
\]
\end{proof}

We now count how often the large-value sets of many almost diametral
functions can overlap. The same-sign overlap estimate implies that these
sets cannot overlap too often.

\begin{lemma}\label{lem:loc_3}
Let $f_1,\ldots,f_N\in B_{L_1(\mu)}$ with $\|f_j-f_k\|_1>2-\beta$ for
$j\neq k$. Given a measurable set $H\subset\Omega$ and $\theta>0$, define
$A_j = H\cap \{|f_j|>\theta\}$ for $j=1,\ldots,N$. Then
\[
        \sum_{j=1}^N \mu(A_j)
        \leq 2\mu(H)+\frac{N(N-1)\beta}{4\theta}.
\]
\end{lemma}

\begin{proof}
Split $A_j=A_j^+ \cup A_j^-$, where
\[
        A_j^+=H\cap\{f_j>\theta\},
        \qquad
        A_j^-=H\cap\{f_j<-\theta\}.
\]
A simple combinatorial argument shows that for any $t\in H$,
\begin{equation}\label{eq:loc_3_1} \sum_{j=1}^N \mathbf{1}_{A_j^+}(t) \leq 1 + \sum_{j<k} \mathbf{1}_{A_j^+ \cap A_k^+}(t).
\end{equation}
Indeed, the left term counts how many sets $A_j^+$ contain the point $t$.
Denote such number by $n_t$. If $n_t=0$ or $n_t=1$, then the inequality is
trivial. Finally, if $n_t\geq 2$, then $t$ belongs to $\binom{n_t}{2}$
intersections of the form $A_j^+ \cap A_k^+$ with $j<k$. Since
$n_t \leq 1 + \binom{n_t}{2}$, the inequality follows. Integrating
(\ref{eq:loc_3_1}) gives
\begin{equation}\label{eq:loc_3_2}
       \sum_{j=1}^N \mu(A_j^+) \leq \mu(H) + \sum_{j<k} \mu(A_j^+ \cap A_k^+).
\end{equation}
Moreover, observe that
\begin{equation}\label{eq:loc_3_3}
	\mu(A_j^+ \cap A_k^+) = \int_{A_j^+ \cap A_k^+} 1 \, d\mu \leq \frac{1}{\theta} \int_{A_j^+ \cap A_k^+} \min(f_j^+,f_k^+) \, d\mu.
\end{equation}
Furthermore, the same estimates hold for the negative sets $A_j^-$. Therefore,
combining (\ref{eq:loc_3_2}) and (\ref{eq:loc_3_3}), and applying
Lemma \ref{lem:loc_1} gives
\[ \sum_{j=1}^N \mu(A_j) = \sum_{j=1}^N \mu(A_j^+) + \sum_{j=1}^N \mu(A_j^-) \leq 2\mu(H) + \sum_{j<k} \frac{\beta}{2\theta} = 2\mu(H) + \frac{N(N-1)\beta}{4\theta}.\qedhere\]
\end{proof}

The counting estimate now yields the desired localization. We first choose
the threshold \(\theta\) small enough so that the functions are small outside
their large-value sets. Then we choose \(N\) large and the separation error
\(\beta\) small so that two large-value sets occupy a set of small control
measure.

\begin{lemma}\label{lem:loc_4}
	Let \(\nu(A)=\int_A |w|\,d\mu\), where \(w\in L_1(\mu)\). For every
	\(0<\alpha<2\) and every \(\eta>0\), there exist an integer \(N\geq2\) and
	a number \(\beta>0\) such that whenever \(f_1,\ldots,f_N\in B_{L_1(\mu)}\)
	satisfy
	\[
	\|f_j-f_k\|_1>2-\beta
	\quad \forall\, j\ne k,
	\]
	there are indices \(j\ne k\) and a measurable set \(E\subset\Omega\) such
	that
	\[
	\nu(E)<\eta
	\qquad\text{and}\qquad
	\int_E |f_j-f_k|\,d\mu>\alpha.
	\]
\end{lemma}

\begin{proof}
If $\nu(\Omega)<\eta$, take $N=2$, $\beta=2-\alpha$, and $E=\Omega$; the
conclusion is immediate. We may therefore assume $\nu(\Omega)\geq\eta$.
Applying Lemma~\ref{lem:loc_2}, there are $H\subseteq \Omega$ with
$\mu(H)<\infty$ and $M\geq 1$ such that
\[
        \nu(H^c)<\frac{\eta}{2},
        \quad \text{and} \quad
        \nu(A)\leq M\mu(A) \quad \forall\, A\subset H.
\]
If $\mu(H)=0$, then $\nu(\Omega)<\eta/2$, contrary to the present case. Thus
$\mu(H)>0$. Choose $\theta>0$, $N\geq 2$, and $\beta>0$ satisfying
\[
        \theta < \frac{2-\alpha}{4\mu(H)}, \qquad N\geq \frac{12M\mu(H)}{\eta}, \qquad \beta\leq \min\left \{\frac{2-\alpha}{2},\frac{4\theta \mu(H)}{N(N-1)}\right \}.
\]
Let $f_1,\ldots,f_N\in B_{L_1(\mu)}$ with $\|f_j-f_k\|_1>2-\beta$ for
$j\neq k$, and recall that $A_j=H\cap\{|f_j|>\theta\}$.
By Lemma~\ref{lem:loc_3},
\[
        \sum_{j=1}^N \mu(A_j)
        \leq 2\mu(H) + \frac{N(N-1)\beta}{4\theta} \leq 3\mu(H),
\]
where we have used the choice of $\beta$. Choose two indices \(j\ne k\) for which
\(\mu(A_j)\) and \(\mu(A_k)\) are minimal. Then
\[
\mu(A_j\cup A_k)
\leq
\mu(A_j)+\mu(A_k)
\leq
\frac{2}{N}\sum_{\ell=1}^{N}\mu(A_\ell)
\leq
\frac{6\mu(H)}{N}\leq \frac{\eta}{2M},
\]
where we have used the choice of $N$ in the last inequality. Set
\[
        E=H^c\cup A_j\cup A_k.
\]
Then
\[
        \nu(E) \leq \nu(H^c)+\nu(A_j \cup A_k) < \frac{\eta}{2} + M\mu(A_j\cup A_k)\leq \eta.
\]
Finally, on $E^c\subset H\setminus(A_j\cup A_k)$ we have $|f_j|\leq\theta$
and $|f_k|\leq\theta$. Therefore
\begin{align*}
        \int_{E}|f_j-f_k|\,d\mu& = \|f_j-f_k\|_1 - \int_{E^c} |f_j-f_k|\, d\mu \geq 2-\beta - \int_{E^c}(|f_j|+|f_k|)\,d\mu
        \\
        &\geq 2-\beta -2\theta\mu(H) > 2-\frac{2-\alpha}{2} -\frac{2-\alpha}{2}=\alpha,
\end{align*}
where we have used the choices of $\theta$ and $\beta$.
\end{proof}

With Lemma \ref{lem:loc_4} in hand, Proposition \ref{prop:G-two-dense} follows immediately from Lemma \ref{lem:localization-to-density}.

\begin{proof}[Proof of Proposition \ref{prop:G-two-dense}]
	Fix $\eta>0$, $w \in L_1(\mu)$, and $0<\alpha<2$. Define $\nu(A)=\int_A |w|\, d\mu$. By Lemma \ref{lem:loc_4}, there is an integer $N\geq2$ and $\beta>0$ such that whenever $f_1,\ldots,f_N \in B_{L_1(\mu)}$ satisfy
	\[ \|f_j - f_k\|_1 > 2-\beta \quad \forall \, j \neq k,\]
	there are indices $j\neq k$ and a measurable set $E \subseteq \Omega$ such that
	\[ \nu(E) <\eta \qquad \text{and} \qquad \int_{E} |f_j-f_k| \, d\mu > \alpha. \]
	Since $C$ is wHD2, Lemma \ref{lem:diametral} guarantees the existence of $f_1,\ldots,f_N$ in $C$ satisfying the required separation condition. Applying Lemma \ref{lem:localization-to-density} we obtain that $G(C,\alpha)$ is weak-star open and weak-star dense.
\end{proof}

\subsection{\texorpdfstring{\(L_1\)-localization for wHD-$\delta$ sets}
	{L1-localization for wHD-delta sets}}

The combinatorial localization argument used in the diameter two case is an
endpoint argument. It relies on the fact that two points of \(B_{L_1(\mu)}\)
which are almost at distance two have very small same-sign overlap. For a
general wHD-\(\delta\) set, with \(\delta<2\), one cannot expect to find
families which are almost at distance two. We therefore use a different
principle. The idea is that if the desired localization fails, then the
differences \(y-z\), with \(y,z\in C\), are almost uniformly integrable with
respect to the prescribed control measure. After a truncation, this puts
\(C\) inside a small norm-neighborhood of a weakly compact set.
A quantitative form of dentability then gives a weakly open piece of \(C\)
of small diameter, contradicting the wHD-\(\delta\) assumption.

We now recall the quantitative form of dentability used below. Let \(A\) be a
bounded subset of a Banach space \(Y\). The \emph{De Blasi measure of weak noncompactness} of \(A\) is
\[
\omega(A)
=
\inf\left\{
r>0:
A\subset K+rB_Y
\text{ for some weakly compact }K\subset Y
\right\}.
\]
We also use the \emph{double-limit measure}
\[
\gamma(A)
=
\sup
\left|
\lim_n\lim_m y_m^*(a_n)
-
\lim_m\lim_n y_m^*(a_n)
\right|,
\]
where the supremum is taken over all sequences \((a_n)\subset A\) and
\((y_m^*)\subset B_{Y^*}\) for which the two iterated limits exist. These
two quantities were studied by Angosto and Cascales in
\cite[Definition~1]{AngostoCascales2009}. We shall use their estimate
\[
\gamma(A)\leq 2\omega(A),
\]
which follows from \cite[Theorem~2.3]{AngostoCascales2009}. We also need a quantitative dentability index. If \(A\subset Y\) is bounded,
define
\[
\operatorname{Dent}(A)
=
\inf\left\{
\operatorname{rad}(S):
S \text{ is a slice of } A
\right\},
\]
where
\[
\operatorname{rad}(S)
=
\inf_{y\in Y}\sup_{s\in S}\|s-y\|.
\]
Here a slice of \(A\) means a non-empty set of the form
\[
S(A,y^*,a)
:=
\left\{y\in A:y^*(y)>\sup_A y^* - a\right\},
\qquad y^*\in Y^*,\ a>0.
\]
This index was introduced by Cascales, P\'erez and Raja in
\cite[Definition~4.1]{CascalesPerezRaja2014}. We shall use two facts from
their paper. First, \cite[Proposition~4.5]{CascalesPerezRaja2014} gives that for every bounded set \(A\subset Y\),
\[
\operatorname{Dent}(A)
=
\operatorname{Dent}(\operatorname{co}(A))
=
\operatorname{Dent}\left (\overline{\operatorname{co}(A)}\right ).
\]
Second, if \(C\subset Y\) is bounded and convex, then \cite[Proposition~4.20]{CascalesPerezRaja2014} yields
\[
\operatorname{Dent}(C)\leq \gamma(C).
\]
The following lemma is the only consequence of these quantitative notions that we shall use.

\begin{lemma}\label{lem:approx-dentability}
	Let \(Y\) be a Banach space and let \(A\subset Y\) be non-empty and bounded. Suppose that
	there are a weakly compact set \(K\subset Y\) and a number \(r>0\)
	such that
	\[
	A\subset K+rB_Y.
	\]
	Then, for every \(\varepsilon>0\), there exists a weakly open set
	\(W\subset Y\) such that \(A\cap W\ne\emptyset\) and
	\[
	\diam(A\cap W)<4r+\varepsilon.
	\]
\end{lemma}

\begin{proof}
	Replacing \(K\) by its weakly closed convex hull, which is again weakly compact, we may assume that \(K\) is convex. Since \(A\subset K+rB_Y\), we have
	\[
	\operatorname{co}(A)\subset K+rB_Y.
	\]
	Therefore $\omega(\operatorname{co}(A))\leq r$. Using the quoted estimates, we get
	\[
	\begin{aligned}
		\operatorname{Dent}(A)
		&=
		\operatorname{Dent}(\operatorname{co}(A))
		\leq
		\gamma(\operatorname{co}(A))
		\leq
		2\omega(\operatorname{co}(A))
		\leq
		2r.
	\end{aligned}
	\]
	Hence there is a slice \(S\) of \(A\) such that
	\[
	\operatorname{rad}(S)<2r+\frac{\varepsilon}{2}.
	\]
	Since $\diam(S)\leq 2\operatorname{rad}(S)$,
	we obtain
	\[
	\diam(S)<4r+\varepsilon.
	\]
	Writing the slice as
	\[
	S=\{y\in A:y^*(y)>\sup_A y^*-a\},
	\]
	and setting
	\[
	W=\{y\in Y:y^*(y)>\sup_A y^*-a\},
	\]
	we have that \(W\) is a weakly open set, \(A\cap W=S\ne\emptyset\),
	and
	\[
	\diam(A\cap W)<4r+\varepsilon.\qedhere
	\]
\end{proof}

The next lemma is the \(L_1\)-specific part of the argument. It says that if
differences of points of \(C\) cannot put much mass on sets which are small
for a fixed \(L_1\)-control measure, then \(C\) is close to a
weakly compact set.

\begin{lemma}\label{lem:close-to-weakly-compact}
	Let \(C\subset B_{L_1(\mu)}\), let $w \in L_1(\mu)$, and define $\nu(A)=\int_A |w|\, d\mu$.
	Suppose that for some \(r>0\) and some \(\eta_0>0\), one has
	\[
	\int_E |y-z|\,d\mu\leq r
	\]
	for all \(y,z\in C\) and every measurable set \(E\subset\Omega\) with
	\(\nu(E)<\eta_0\). Then there exists a weakly compact set
	\(K\subset L_1(\mu)\) such that
	\[
	C\subset K+rB_{L_1(\mu)}.
	\]
\end{lemma}

\begin{proof}
	If \(C=\emptyset\), there is nothing to prove. Hence assume \(C\ne\emptyset\). Lemma \ref{lem:loc_2} provides a measurable set $H \subset \Omega$ with $\mu(H)<\infty$ and a constant $M\geq 1$ such that
	\[
	\nu(H^c)<\frac{\eta_0}{2} \qquad \text{and} \qquad \nu(A)\leq M\mu(A) \quad \forall \, A \subset H \text{ measurable}.
	\]
	Choose \(R>0\) such that
	\[
	\frac{2M}{R}<\frac{\eta_0}{2}.
	\]
	Let \(h\in C-C\). Define
	\[
	E_h=H^c\cup\bigl(H\cap\{|h|>R\}\bigr).
	\]
	Since \(h\in 2B_{L_1(\mu)}\), we have
	\[
	\begin{aligned}
		\nu\bigl(H\cap\{|h|>R\}\bigr)
		&
		\leq
		M\,\mu\bigl(H\cap\{|h|>R\}\bigr)
		\leq
		M\frac{\|h\|_1}{R}
		\leq
		\frac{2M}{R}
		<
		\frac{\eta_0}{2}.
	\end{aligned}
	\]
	Thus \(\nu(E_h)<\eta_0\), and by hypothesis
	\[
	\int_{E_h}|h|\,d\mu\leq r.
	\]
	Consequently, we can approximate $h$ by $h_R:=h\,\mathbf 1_{H\cap\{|h|\leq R\}}$. Indeed,
	\[
	\|h-h_R\|_1=\int_{E_h}|h|\,d\mu\leq r.
	\]
	Viewed as a subset of \(L_1(\mu|_H)\), the family
	\[
	K_0=\{h_R:h\in C-C\}
	\]
	is uniformly integrable, since \(\mu(H)<\infty\) and
	\[
	|h_R|\leq R\mathbf 1_H
	\quad \forall \, h\in C-C.
	\]
	By the Dunford--Pettis criterion, \(K_0\) is relatively weakly compact in
	\(L_1(\mu|_H)\). The extension-by-zero operator
	\[
	L_1(\mu|_H)\longrightarrow L_1(\mu)
	\]
	is bounded linear and therefore weak-to-weak continuous. Hence \(K_0\) is
	relatively weakly compact in \(L_1(\mu)\). Fix $c_0 \in C$. It follows that
	\[K=c_0+\overline{K_0}^{\,w}\] 
	is weakly compact in $L_1(\mu)$. Moreover, for any \(c\in C\), we have that $h=c-c_0 \in C-C$, and
	\[
	c=c_0+h=c_0+h_R+(h-h_R),
	\qquad
	\|h-h_R\|_1\leq r.
	\]
	Therefore
	\[
	C\subset K+rB_{L_1(\mu)}.\qedhere
	\]
\end{proof}

We can now prove the desired localization principle.

\begin{proposition}\label{prop:positive-localization}
	Let \(0<\delta\leq2\), let \(C\subset B_{L_1(\mu)}\) be wHD-\(\delta\), and
	let \(0<\alpha<\delta/4\). Then for every \(w\in L_1(\mu)\) and every
	\(\eta>0\), there exist \(y,z\in C\) and a measurable set
	\(E\subset\Omega\) such that
	\[
	\int_E |w|\,d\mu<\eta
	\qquad\text{and}\qquad
	\int_E |y-z|\,d\mu>\alpha.
	\]
\end{proposition}

\begin{proof}
	Suppose, by contradiction, that the conclusion fails for some
	\(w\in L_1(\mu)\) and some \(\eta_0>0\). Define $\nu(A)=\int_A |w|\, d\mu$. Then
	\[
	\int_E |y-z|\,d\mu\leq \alpha
	\]
	for all \(y,z\in C\) and every measurable set \(E\subset\Omega\) with $\nu(E)<\eta_0$.
	By Lemma~\ref{lem:close-to-weakly-compact}, there is a weakly compact set \(K\subset L_1(\mu)\) such that
	\[
	C\subset K+\alpha B_{L_1(\mu)}.
	\]
	Applying Lemma~\ref{lem:approx-dentability} with \(A=C\) and \(r=\alpha\),
	we obtain, for every \(\varepsilon>0\), a weakly open set
	\(W\subset L_1(\mu)\) such that \(C\cap W\ne\emptyset\) and
	\[
	\diam(C\cap W)<4\alpha+\varepsilon.
	\]
	Since \(\alpha<\delta/4\), we may choose \(\varepsilon>0\) so small that
	\[
	4\alpha+\varepsilon<\delta,
	\]
	which contradicts the fact that \(C\) is wHD-\(\delta\).
\end{proof}

Now that the desired $L_1$-localization has been established, Proposition~\ref{prop:G-delta-dense} follows from Lemma \ref{lem:localization-to-density}.

\begin{proof}[Proof of Proposition~\ref{prop:G-delta-dense}]
	Let \(0<\delta\leq2\), let \(C\subset B_X\) be wHD-\(\delta\), and let
	\(0<\alpha<\delta/4\). Proposition~\ref{prop:positive-localization} verifies
	the hypothesis of Lemma~\ref{lem:localization-to-density}. Therefore
	\(G(C,\alpha)\) is weak-star open and weak-star dense in \(B_{X^*}\).
\end{proof}

\section*{Acknowledgements}
The author thanks Gin\'es L\'opez-P\'erez for helpful comments and suggestions.

\bibliographystyle{plain}
\bibliography{bibliography}
\end{document}